\documentclass[11pt]{article}
\usepackage[latin1]{inputenc}
\usepackage{amsmath}
\usepackage{amsfonts}
\usepackage{amssymb}
\usepackage[american]{babel}
\usepackage{graphicx}
\usepackage{amsthm}
\usepackage{pstricks}

\newtheorem{proposition}{Proposition}
\newtheorem{theorem}{Theorem}
\newtheorem{lemma}{Lemma}
\newtheorem{problem}{Problem}
\newtheorem{definition}{Definition}

\newtheorem{conjecture}{Conjecture}
\newcommand{\s}{\v{s}}

\theoremstyle{definition}

\begin{document}

\title{A family of multigraphs with large palette index}

\author{M.Avesani \thanks{Dipartimento di Informatica,
Universit\`a di Verona, Strada Le Grazie 15, Verona (Italy)}, 
A.Bonisoli \thanks{Dipartimento di Scienze Fisiche, Informatiche e Matematiche, 
Universit\`a di Modena e Reggio Emilia, Via Campi 213/b, Modena (Italy)}, 
G.Mazzuoccolo \thanks{Dipartimento di Informatica, Universit\`a di Verona, 
Strada Le Grazie 15, Verona (Italy)}}

\maketitle

\begin{abstract}
\noindent 
Given a proper edge-coloring of a loopless multigraph, the palette of a vertex 
is defined as the set of colors of the edges which are incident with it. 
The palette index of a multigraph is defined as the minimum number of distinct 
palettes occurring among the vertices, taken over all proper edge-colorings 
of the multigraph itself. In this framework, the palette multigraph of an 
edge-colored multigraph is defined in this paper and some of its properties 
are investigated. 
We show that these properties can be applied in a natural way in order to produce the first known family
of multigraphs whose palette index is expressed in terms of the maximum degree by a quadratic polynomial.
We also attempt an analysis of our result in connection with some related questions.
\end{abstract}

\noindent 
\textit{Keywords: palette index, edge-coloring, interval edge-coloring. 
}

\section{Introduction}
\label{sec:intro}

Generally speaking, as soon as a chromatic parameter for graphs is 
introduced, the first piece of information that is retrieved is 
whether some universal meaningful upper or lower bound holds for it. 
This circumstance is probably best exemplified by mentioning, say, 
Brooks' theorem for the chromatic number and Vizing's theorem for 
the chromatic index. In either instance the maximum degree $\Delta$ 
is involved and that probably explains the trend to consider $\Delta$ 
as a somewhat natural parameter, in terms of which bounds for other 
chromatic parameters should be expressed. In the current paper we 
make no exception to this trend and use the maximum degree $\Delta$ as a 
reference value for the recently introduced chromatic parameter known 
as the \textit{palette index}. To this purpose we introduce an additional 
tool, that we call the \textit{palette multigraph}, which can be defined 
from a given loopless multigraph with a a proper edge-coloring. Some 
properties of the palette multigraph are investigated in Section 
\ref{sec:preliminaries} and we feel they might be of interest in their
own right. In the current context, we use these properties in connection 
with an attempt of finding a polynomial upper bound in terms of $\Delta$ 
for the palette index of a multigraph with maximum degree $\Delta$.  
As a consequence of our main construction in Section \ref{sec:mainconstruction},
we can assert that if such a polynomial bound exists at all then it must 
be at least quadratic.

Throughout the paper we use the term \textit{multigraph} 
to denote an undirected multigraph with no loops. 
For any given multigraph $G$, we always denote by 
$V(G)$ and $E(G)$ the set of vertices and the set of 
edges of $G$, respectively.  We further denote by $G_s$ 
the simple graph obtained from $G$ by shrinking to a single 
edge any set of multiple edges joining two given vertices.
  
By a \textit{coloring} of a multigraph $G$ we always mean 
a proper edge-coloring of $G$. A coloring of $G$ is thus 
a mapping $c: E(G)\to C$, where $C$ is a finite set whose 
elements are designated as colors, with the property that 
adjacent edges always receive distinct colors. We shall 
often say that $(G,c)$ is a colored multigraph, meaning 
that $c$ is a coloring of the multigraph $G$.

Given a colored multigraph $(G,c)$, the 
\textit{palette} $P_c(x)$ of a vertex $x$ of $G$ is the 
set of colors that $c$ assigns to the edges which are 
incident with $x$.

The palette index $\check s(G)$ of a simple graph $G$ is 
defined in \cite{HKMW} as the minimum number of palettes
occurring in a coloring of $G$. The definition can be extended 
verbatim to multigraphs. The exact value of the palette 
index is known for some classes of simple graphs.

   \begin{itemize}
   \item
A graph has palette index $1$ if and only if
it is a class $1$ regular graph \cite[Proposition 1]{HKMW}.
   \item 
A connected class $2$ cubic graph has palette index $3$ or $4$ 
according as it does or it does not possess a perfect matching, 
respectively \cite[Theorem 9]{HKMW}.
\item 
If $n$ is odd, $n\ge3$ then $\check s(K_n)$ is $3$ or $4$ 
depending on $n\equiv3$ or $1\mod{4}$, respectively 
\cite[Theorem 4]{HKMW}.
   \item 
The palette index of complete bipartite graphs was determined in 
\cite{HH} in many instances.
   \end{itemize}

The quoted result for complete graphs shows that it is possible 
to find a family of graphs, for which the maximum degree can become 
arbitrarily large, and yet the palette index admits a constant upper bound,
namely $4$ in this case.

As it was remarked in \cite{BonMaz}, the fact that a class $2$ regular graph 
of degree $\Delta$ always admits a $(\Delta+1)$-coloring forces  
$\Delta+1$  to be an upper bound for the palette index of such a a graph
($\Delta+1$ is namely the number of $\Delta$-subsets of a $(\Delta+1)$-set 
of colors). 

That is definitely not the case for non-regular graphs: 
it was shown in \cite{BonBonMaz} that for each positive integer 
$\Delta$ there exists a tree with maximum degree $\Delta$ whose palette
index grows asymptotically as $\Delta\ln(\Delta)$.

Consequently, one cannot expect for the palette index any analogue of, say, 
Vizing's theorem for the chromatic index: the palette index of graphs of 
maximum degree $\Delta$ cannot admit a linear polynomial in $\Delta$ as a 
universal upper bound.

It is the main purpose of the present paper to produce an infinite family 
of multigraphs, whose palette index grows asymptotically as $\Delta^2$, see 
Section \ref{sec:mainconstruction}. Our method relies essentially on a tool 
that we define in Section \ref{sec:preliminaries}, namely the palette multigraph
of a colored multigraph. This concept is strictly related to the notion 
of palette index and it appears to yield a somewhat natural approach to the study 
of this chromatic parameter.

\section{The palette multigraph of a colored multigraph}
\label{sec:preliminaries} 

For any given finite set $X$ and positive integer $t$ we denote 
by $t\cdot X$ the multiset in which each element of $X$ is repeated $t$ times.

The next definition will play a crucial role for our  
construction in Section \ref{sec:mainconstruction}.
Given a colored multigraph  $(G,c)$, we  
define its \textit{palette multigraph} $\Gamma_c(G)$
as follows.

The vertex-set of $\Gamma_c(G)$ is 
$V(\Gamma_c(G))=\{ P_c(v): v \in V(G)\}$.
In other words the vertices of $\Gamma_c(G)$ are 
the palettes of $(G,c)$. 

For any given pair of adjacent vertices $x$ and $y$ of $G$, 
we declare the (not necessarily distinct) palettes 
$P_c(x)$ and $P_c(y)$ to be adjacent and define the 
corresponding edge in the palette multigraph $\Gamma_c(G)$.

More precisely, if $x$ and $y$ are adjacent vertices 
in $G$ such that their palettes $P_c(x)$ and $P_c(y)$ 
are distinct, then $P_c(x)$ and $P_c(y)$ yield two 
distinct vertices connected by an ordinary edge in the 
palette multigraph $\Gamma_c(G)$, see vertices $x_1$ and $x_2$ in Figure
\ref{fig:palettegraph}. If, instead, 
$x$ and $y$ are adjacent vertices in $G$ with 
equal palettes $P_c(x)$ and $P_c(y)$, these form a single 
vertex with a loop in the palette multigraph 
$\Gamma_c(G)$, see vertices $x_2$ and $x_3$ in Figure 
\ref{fig:palettegraph}.

If two (equal or unequal) palettes appear on several pairs 
of adjacent vertices of $G$, then each such pair yelds one edge 
in $\Gamma_c(G)$ (either a loop or an ordinary edge).
It is thus quite possible that the palette multigraph $\Gamma_c(G)$
presents multiple (ordinary) edges between two given distinct vertices
as well as multiple loops at a given vertex.

An example of a pair $(G,c)$ and the corresponding palette multigraph 
$\Gamma_c(G)$ is presented in Figure \ref{fig:palettegraph}.

\begin{figure}
\centering
\includegraphics[width=12cm]{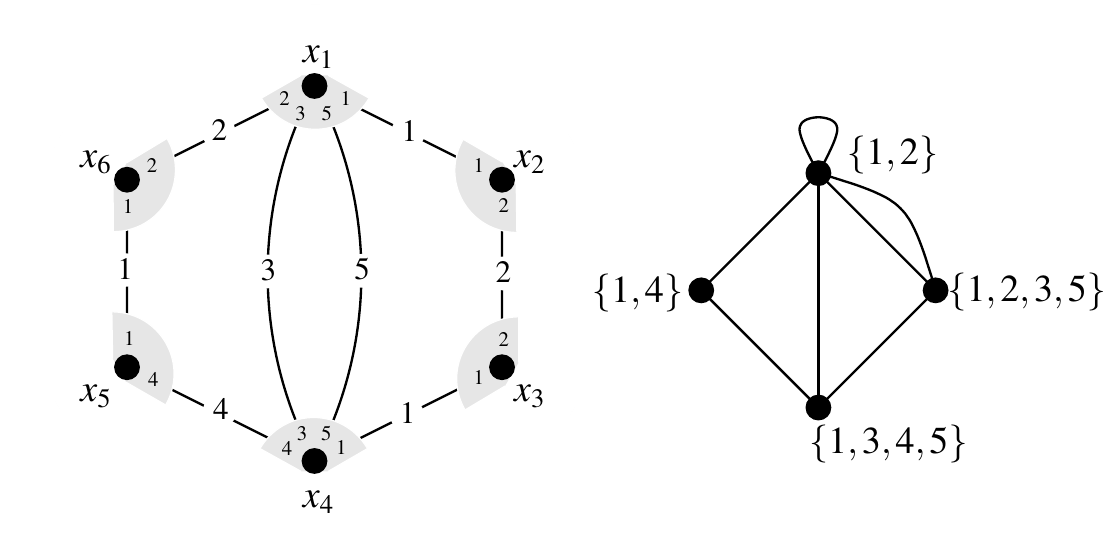}
\caption{A multigraph $G$ with a proper edge-colouring and the associated palette multigraph} 
\label{fig:palettegraph}
\end{figure}

The number of vertices of the palette multigraph $\Gamma_c(G)$ is thus 
equal to the number of distinct palettes in the colored multigraph $(G,c)$, while 
the number of edges (loops and ordinary edges) in $\Gamma_c(G)$ is equal to the 
number of edges in the underlying simple graph $G_s$.

The following proposition is also an easy consequence of the definition of the 
palette multigraph: Note that each loop in $\Gamma_c(G)$ contributes $2$ 
to the degree of its vertex. 

\begin{proposition}\label{degrees_palettegraph}
For any given colored multigraph $(G,c)$, 
the degree of a vertex $P_c(x)$ in the palette multigraph 
$\Gamma_c(G)$ is equal to the sum of the degrees in the 
underlying simple graph $G_s$ of all vertices whose palette 
in $(G,c)$ is equal to $P_c(x)$.  
\end{proposition}

\section{The main construction}
\label{sec:mainconstruction} 

The main purpose of this Section is the construction of a 
multigraph $G^{\Delta}$ with maximum degree $\Delta$, 
whose palette index is expressed by a quadratic polynomial 
in $\Delta$.

For the sake of brevity we shall assume $\Delta$ even, $\Delta\ge2$:
a slight modification of our construction yields the 
same result for odd values of $\Delta$.

The multigraph $G^{\Delta}$ is obtained as the disjoint union of multigraphs 
$H^{\Delta}_t$, for $t=1,2,...,\Delta-2$, which are defined as follows.

Let $H^{\Delta}$ be the simple graph with vertices $u$, $v^0$, $v^1$,
\ldots, $v^{\Delta-1}$ and edges $uv^0$, $uv^1$,
\ldots, $uv^{\Delta-1}$, $v^0v^1$, $v^2v^3$, \ldots, $v^{\Delta-2}v^{\Delta-1}$.
The graph $H^{\Delta}$ is sometimes called a windmill graph \cite{Gal}
and can also be described as being obtained from the wheel $W_{\Delta}$ 
(see \cite{BonMur}) by alternately deleting edges on the outer cycle.

The multigraph $H^{\Delta}_t$ is obtained by replacing each edge
$v^jv^{j+1}$ which is not incident with the central vertex $u$ 
with $t$ repeated edges between the same vertices $v^j$ and $v^{j+1}$.

In detail, define for $t=1,2,...,\Delta-2$
\begin{equation*}
V(H^\Delta _t)=\{u_t, v^0_t, v^1_t,\ldots, v^{\Delta-1}_t\}
\end{equation*}
\begin{equation*}
E(H^\Delta _t)=t \cdot\{v^j_tv^{j+1}_t : j\in \{0,2,4,...\Delta-2\}\} \cup \{ u_tv^{j}_t : j\in \{0,1,2,...\Delta-1\}\}
\end{equation*}
\begin{equation*}
H^{\Delta}_t=\left(V(H^{\Delta}_t),E(H^{\Delta}_t)\right)
\end{equation*}
For $j=0,1,\ldots,\Delta-1$ we denote the edge $u_tv^j_t$ by $e^j_t$ or simply by $e^j$ 
once $t$ is understood. Furthermore, for any index $j\in\{0,1,\ldots,\Delta-1\}$ there 
is a uniquely determined index $j'\in\{0,1,\ldots,\Delta-1\}$, $j\ne j'$ such that 
$v^{j'}_t$ is the unique vertex, other than $u_t$, which is adjacent to 
$v^j_t$ in $H^{\Delta}_t$.

The submultigraph of $H^{\Delta}_t$ which is induced by the vertices 
$u_t$, $v^j_t$, $v^{j'}_t$ will be denoted by $L^{\,j}$. The edges of 
$L^{\,j}$ are $e^j$, $e^{j'}$ and the $t$ repeated edges 
having  $v^j_t$ and $v^{j'}_t$ as endvertices.
By definition, we have $L^{\,j}=L^{\,j'}$.

\begin{figure}
\centering
\includegraphics[width=6cm]{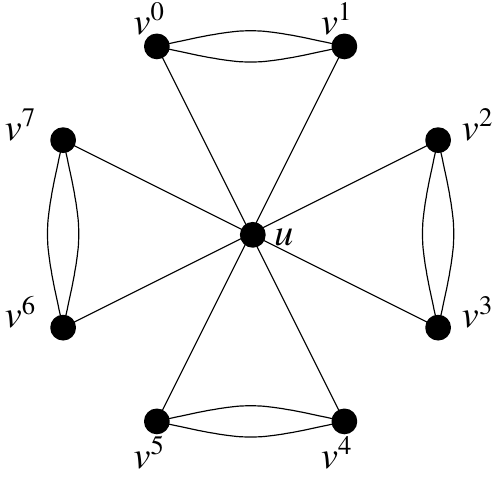}
\caption{The graph $H^8_2$} 
\label{fig:H82}
\end{figure}

We now assume that a $k$-edge-coloring $c:E(H^\Delta _t)\to C=\{c_0,c_1,\ldots,c_{k-1}\}$
is given and study some properties of the palette multigraph $\Gamma_c(H^{\Delta}_t)$.
Since the central vertex $u_t$ has degree $\Delta$ in $H^{\Delta} _t$ 
we have $\Delta\le k$ and may assume, with no loss of generality, that  
$c(e^j)=c_j$ holds for $j=0,1,\ldots,\Delta-1$. The inequality $t\le\Delta-2$ yields 
in turn $t+1<\Delta$. Consequently, since each non central vertex $v^j_t$ has 
degree $t+1$, we see that the palette $P_c(u_t)=\{0,1,\ldots,\Delta-1\}$ is distinct 
from every other palette $P_c(v^j_t)$. For that reason, rather than looking at the 
palette multigraph $\Gamma_c(H^{\Delta}_t)$ we consider the submultigraph 
$\Gamma^-_c(H^{\Delta}_t)=\Gamma_c(H^{\Delta}_t)\smallsetminus P_c(u_t)$ 
obtained by removing the palette $P_c(u_t)$ (as a vertex of the palette multigraph).

\begin{lemma}
The multigraph $\Gamma^-_c(H^{\Delta}_t)$ is a simple graph and is a forest.
\end{lemma}
\begin{proof}
We prove first of all that $\Gamma^-_c(H^{\Delta}_t)$ has no loop, 
that is $P_c(v^j_t) \ne P_c(v^{j'}_t)$ for all $j$.
Consider the two adjacent vertices $v^j_t$ and $v^{j'}_t$. 
The corresponding edges $e^j$ and $e^{j'}$ have distinct colors 
$c_j$ and $c_{j'}$ in $\{0,.....,\Delta-1\}$, respectively. The color 
$c_j$ cannot appear on one of the edges between $v^j_t$ and $v^{j'}_t$, 
since $c$ is a proper coloring. Hence, $c_j$ belongs to $P_c(v^j_t)$ and 
does not belong to $P_c(v^{j'}_t)$, and the two palettes are distinct, as claimed.

Next, we prove that $\Gamma^-_c(H^{\Delta}_t)$ has no multiple edges, by showing that 
if $P_c(v^j_t)=P_c(v^h_t)$ for $h\ne j,j' $, then $P_c(v^{j'}_t)\ne P_c(v^{h'}_t)$.
Suppose the vertices $v^j_t$ and $ v^h_t$ share the same palette. The edges $e^j$ and 
$e^h$ are colored with colors $c_j$ and $c_h$, respectively. Hence $\{c_j,c_h\}\subset P_c(v^j_t)(=P_c(v^h_t))$. 
In particular, one of the edges between $v^h_t$ and $v^{h'}_t$ has color $c_j$ 
and so we have $c_j\in P_c(v^{h'}_t)$. On the other hand, $c_j$ does not belong to 
$P_c(v^{j'}_t)$ because $c$ is a proper coloring, and the claim follows.  


In order to complete our proof, we need to prove that 
$\Gamma^-_c(H^{\Delta}_t)$ has no cycle and is thus a forest. 

Assume, by contradiction, that $\Gamma^-_c(H^{\Delta}_t)$ has a cycle $\Gamma$. 
Without loss of generality, we may assume that $\Gamma$ contains the vertices 
$P_c(v^0_t)$ and $P_c(v^1_t)$ of $\Gamma^-_c(H^{\Delta}_t)$. Since $P_c(v^0_t)$ 
has degree at least two in $\Gamma^-_c(H^{\Delta}_t)$, there exists $h\ne0$ such 
that $P_c(v^0_t)=P_c(v^h_t)$ and $P_c(v^{h'}_t)$ belongs to $\Gamma$. Recall that $e_0$ 
has colour $c_0$ in $c$. Therefore, the colour $c_0$ belongs to both palettes $P_c(v^0_t)$
and $P_c(v^h_t)$, since they are the same palette. Furthermore, the edge $e_h$ has colour $c_h$, 
different from $c_0$. Then $c_0$ is the colour of one of the edges between $v^h_t$ and 
$v^{h'}_t$. Hence, the colour $c_0$ also belongs to the palette $P_c(v^{h'}_t)$. 
Repeating the same argument, we obtain that $c_0$ belongs to each palette of the cycle 
$\Gamma$. That is a contradiction, since $c_0$ does not belong to the palette $P_c(v^1_t)$.
\end{proof}

\begin{lemma}\label{lem:average_degree}
The degree of a vertex $P_c(v^j_t)$ in $\Gamma^-_c(H^{\Delta}_t)$ is 
exactly equal to the number of vertices of $G$ having the same palette 
$P_c(v^j_t)$ in the colouring $c$. 
\end{lemma}
\begin{proof}
The underlying simple graph of $H^{\Delta}_t\smallsetminus\{u_t\}$ is 
the disjoint union of isolated edges, that is every vertex has degree exactly $1$ 
in the underlying simple graph. It follows from Proposition \ref{degrees_palettegraph}
that when a given palette $P$ is viewed as a vertex in $\Gamma^-_c(H^{\Delta}_t)$,  
then its degree is equal to the number of vertices in $G$ sharing the palette $P$. 
\end{proof}

The next Proposition states a well-known property of forests.
\begin{proposition}
The average degree of a forest is strictly less than $2$.
\end{proposition}
\begin{proof}
Suppose that the forest $F$ has $n$ vertices. Then $F$ has at most $n-1$ edges and
\begin{equation*}
\sum_{v\in V(F)}d(v)=2|E(F)|\le2(n-1)
\end{equation*}
so that the average degree is
\begin{equation*}
\frac{1}{n}\sum_{v\in V(F)}d(v)\leq \frac{2(n-1)}{n}<2
\end{equation*}
\end{proof}


By the previous Proposition and Lemma \ref{lem:average_degree}, 
the average number of vertices in $H^{\Delta}_t$ sharing the same 
palette is less than $2$ and that implies the following lower bound 
for the palette index of $H^{\Delta}_t$:
\begin{equation*}
\text{\s}(H^{\Delta}_t) > \frac{\Delta}{2} + 1
\end{equation*}



\begin{theorem}
\begin{equation}
\frac{\Delta}{2}\left(\Delta-2\right) < \check s(G^{\Delta}) < \left(\Delta+1\right) \left(\Delta-2\right)
\end{equation}
\end{theorem}
\begin{proof}
The second inequality is an immediate consequence of the fact that 
the number of vertices in $G^{\Delta}$ is 
$\left(\Delta+1\right) \left(\Delta-2\right)$.

For the first inequality, it is sufficient to observe that all vertices of degree $t+1$ 
in $G^\Delta$  belong to the subgraph $H^{\Delta}_t$, so they cannot 
share the same palette with a vertex in another subgraph $H^{\Delta}_{t'}$, 
with $t'\ne t$. On the other hand, the vertex $u_t$ of degree $\Delta$ in 
$H^{\Delta}_t$ could share the same palette with every other vertex of degree 
$\Delta$, one in each subgraph $H^{\Delta}_t$. We obtain 
\[
\check s(G^{\Delta})\ge\sum_t(\check s(H^{\Delta}_t)-1)>\left(\Delta-2\right)\frac{\Delta}{2}.
 \]

\end{proof}

We observe that $G^{\Delta}$ is not connected. If a new vertex $\infty$ 
is introduced and is declared adjacent to each vertex of degree $\Delta$ 
in  $G^{\Delta}$, we obtain a connected multigraph $\widetilde{G}^{\Delta}$ 
of maximum degree $\Delta+1$. The palette index of the multigraphs 
$\widetilde{G}^{\Delta}$ is again bounded from below by a quadratic polynomial in $\Delta$.

\section{Some considerations on a related parameter}

We introduce a new natural parameter related to the palette 
index of a multigraph. Consider an edge-coloring $c$ of $G$ which 
minimizes the number of palettes, that is the number of palettes 
is exactly $\text{\s}(G)$: how many colors does $c$ require? 
More precisely, we consider the minimum $k$ such that there 
exists a $k$-edge-coloring of $G$ with $\text{\s}(G)$ palettes. 
We will denote such a minimum by $\chi'_{\text{\s}} (G)$. 
Obviously, $\chi'_{\text{\s}}(G)\geq \chi'(G)$ because we need at 
least the number of colors in a proper edge-coloring. In \cite{HKMW}, 
the authors remark that in some cases this number is strictly larger 
than the chromatic index of the graph. 
How much larger could it be?\\

An upper bound for the value of $\chi'_{\text{\s}}(G) $ for some classes of graphs 
can be deduced from an analysis of the proofs of the corresponding results for the 
palette index. 

\begin{itemize}
 \item \cite{HKMW} Let $K_n$ be a complete graph with $n>1$ vertices. Then, 
$$\chi'_{\text{\s}}(K_n)= \Delta \quad \text{if} \quad n \equiv 0\mod2 \quad
\chi'_{\text{\s}}(K_n) \leq \frac{3\Delta}{2} \quad \text{if} \quad n \equiv 1\mod2 $$

In particular, if $n=4k+3$ then it is proved that the palette index is equal to $3$ 
and the proof is obtained by using three sets of colors of cardinality $2k+1$. 
If $n=4k+5$, the proof works by using three sets of colors of cardinality $2k+1$ and 
three additional colors, that is $6(k+1)$ colors. The number of colors is exactly 
$\frac{3\Delta}{2}$ in both cases.\\  

\item \cite{HKMW} Let $G$ be a cubic graph. Then,
$$ \chi'_{\text{\s}}(G) \leq 5. $$

In  particular, five colors are necessary if $G$ is not $3$-edge-colorable and 
has no perfect matching.

\item \cite{BonMaz} Let $G$ be a $4$-regular graph. Then,
$$ \chi'_{\text{\s}}(G) \leq 6. $$ 

In particular, six colors are used in some examples with palette index $3$ (see 
the proof of Proposition 11 in \cite{BonMaz}).

\item \cite{BonBonMaz} Let $G$ be a forest. Then,

$$ \chi'_{\text{\s}}(G) = \Delta. $$

\item \cite{HH} Let $K_{m,n}$ be a complete bipartite graph with $1 \leq m \leq n$. 
This situation is a little more involved, in the sense that we cannot always obtain 
a good upper bound for $\chi'_{\text{\s}}(K_{m,n})$ using the proofs of the results 
in \cite{HH}. In some cases, see for instance Proposition 11 in \cite{HH}, the number 
of colors is twice the maximum degree $\Delta$ (recall that minimizing the number of 
colors was not important in that context). 
Nevertheless, we analyze some small cases and obtain 
the same number of palettes (the minimum) by using a smaller number of colors.\\

One such example is obtained by considering the graph $K_{5,6}$ 
(i.e. case $k=3$ in Proposition 11 of \cite{HH}). Denote by $\{u_1,\ldots,u_5\}$ 
and $\{v_1,\ldots,v_6\}$ the bipartition of the vertex-set of $K_{5,6}$. The proof 
of Proposition 11 in \cite{HH} furnishes an edge-coloring with $12$ colors and $6$ palettes. 
Following the notation used in \cite{HH} we represent the coloring with a matrix, where 
the element in position $(i,j)$ is the color of the edge $u_iv_j$.    

\begin{equation*}
M_{5,6}=\begin{pmatrix}
1 & 2 & 3 & 4 & 5 & 6\\
3 & 1 & 2 & 6 & 4 & 5\\
2 & 3 & 1 & 5 & 6 & 4\\
7 & 8 & 9& 10 & 11 & 12\\
8 & 7 & 12 & 11 & 10 & 9\\
\end{pmatrix},
\end{equation*}

The following coloring has only $8$ colors and again $6$ palettes.

\begin{equation*}
M'_{5,6}=\begin{pmatrix}
1 & 2 & 3 & 4 & 5 & 6\\
3 & 1 & 2 & 6 & 4 & 5\\
2 & 3 & 1 & 5 & 6 & 4\\
4 & 5 & 7& 8 & 1 & 2\\
5 & 4 & 8 & 7 & 2 & 1\\
\end{pmatrix},
\end{equation*}

We would like to stress that, even if we can obtain similar colorings 
for some other sporadic cases, we are not able to generalize our results 
to all infinite families considered in \cite{HH}.      

\end{itemize}

All previous results and the study of some sporadic cases suggest 
that $\chi'_{\text{\s}}(G)$ cannot be too large with respect to $\Delta$.
In particular, we believe there exists a linear upper bound for 
$\chi'_{\text{\s}}(G)$ in terms of $\Delta$. The following is thus 
an even stronger conjecture. 

\begin{conjecture}\label{main_conj}
Let $G$ be a (simple) graph. Then,
\begin{equation*}
\chi'_{\text{\s}}(G)\leq \lceil\dfrac{3}{2}\Delta \rceil 
\end{equation*}
\end{conjecture}

As far as we know, this conjecture is new and completely open.
We believe any progress in that direction could be useful 
for a deeper understanding of the behavior of the palette index 
of general graphs. 

\section{Concluding remarks and open problems}

In this final Section we propose some further open questions and 
indicate a few connections with other known problems.

In Section \ref{sec:mainconstruction}, we have presented a family 
of multigraphs whose palette index is expressed by a quadratic 
polynomial in $\Delta$. We were not able to find a family of 
simple graphs with such a property and so we leave the existence 
of such a family as an open problem.

\begin{problem}
For $\Delta=3$, $4$, \ldots, does there exist a simple graph 
with maximum degree $\Delta$ whose palette index is quadratic in $\Delta$?  
\end{problem}

As far as we know, the best general upper bound in terms of $\Delta$ 
for the palette index of a simple graph $G$ is the trivial one, which is 
obtained from a $(\Delta+1)$-edge-colouring $c$ of $G$: in principle, 
each non-empty proper subset of the set of colours could occur as a palette 
of $(G,c)$, whence $\text{\s}(G)\le2^{\Delta+1}-2$. On the other hand, all 
known examples suggest that this upper bound is far from being tight. 
In particular, we raise the question whether a polynomial upper 
bound holding for general multigraphs may exist at all. 

\begin{problem}\label{poly}
Prove the existence of a polynomial $p(\Delta)$ such that 
$\text{\s}(G)\le p(\Delta)$ for every multigraph $G$ with maximum 
degree $\Delta$.
\end{problem}

We slightly suspect that if a polynomial $p$ solving Problem \ref{poly} 
can be found at all, then some quadratic polynomial will do as well.

Finally, we would like to stress how this kind of problems on 
the palette index is somehow related to another well-known type 
of edge-colorings, namely interval edge-colorings, introduced by 
Asratian and Kamalian in \cite{AsrKam}.

\begin{definition}
A proper edge-coloring $c$ of a graph with colors $\{c_1,c_2,\ldots,c_t\}$ 
is called an \emph{interval edge-colouring} if all colours are actually used, 
and the palette of each vertex is an interval of consecutive colors. 
\end{definition}

The following relaxed version of the previous concept was first 
studied in \cite{Kot} and then explicitly introduced in \cite{dWS}.

\begin{definition}
A proper edge-colouring $c$ of a graph with colors 
$\{c_1,c_2,\ldots,c_t\}$ is called an \emph{interval cyclic edge-colouring} 
if all colours are used and the palette of each vertex is either an 
interval of consecutive colors or its complement. 
\end{definition}

Both interval and interval cyclic edge-colorings are thus proper edge-colourings 
with severe restrictions on the set of admissible palettes.   

There are many more results on interval edge-colourings 
(see among others \cite{PetMkh}). In particular, it is known that not 
all graphs admit an interval edge-colouring. Furthermore, it is proved 
in \cite{N} that if a multigraph of maximum degree $\Delta$ admits an 
interval edge-colouring then it also admits an interval cyclic 
$\Delta$-edge-colouring.

The following holds:

\begin{proposition}
Let $G$ be a multigraph of maximum degree $\Delta$ admitting an interval 
edge-colouring. Then, $\text{\s}(G) \le \Delta^2 - \Delta + 1$.
\end{proposition}
\begin{proof}
Since $G$ admits an interval edge-colouring, then it also admits an 
interval cyclic $\Delta$-edge-colouring $c$ (see \cite{N}). 
Each palette of $(G,c)$ is thus an interval of colors in the 
set $\{c_1,c_2,\ldots,c_{\Delta}\}$ or its complement is one such 
interval. For $t=1,\ldots, \Delta-1$, there are exactly $\Delta$ 
such subsets of cardinality $t$, and a unique one for $t=\Delta$. 
We have thus at most $\Delta(\Delta-1)+1$ distinct palettes in $(G,c)$, 
that is $\text{\s}(G) \leq \Delta^2 - \Delta + 1$.
\end{proof}

In other words, the previous Proposition assures that a putative example 
of a family of multigraphs whose palette index grows more than quadratically 
in $\Delta$ should be searched for within the class of multigraphs without 
an interval edge-colouring.\\


In this paper, we also introduce the palette multigraph of a colored 
multigraph $(G,c)$. A precise characterization of the palette multigraph 
of the family introduced in Section \ref{sec:mainconstruction} is the key 
point of our main proof. It suggests that a study of palette multigraphs in 
a general setting could increase our knowledge of the palette index. 
Possibly, it could also help in the search for an answer to some of 
the previous problems. Hence, we conclude our paper with the following:

\begin{problem}
Let $H$ be a multigraph. Determine whether a colored multigraph $(G,c)$ 
exists, such that $H$ is the palette multigraph of $(G,c)$.  
\end{problem}

\end{document}